\documentclass[12pt,reqno]{amsart}
\usepackage{fullpage}
\usepackage{amsmath}
\usepackage{mathrsfs,amssymb,graphicx,verbatim,hyperref}
\usepackage{paralist}

\usepackage{ifthen}

\newcommand{\ifsodaelse}[2]{\ifthenelse{\isundefined{\SODAF}}{#2}{#1}}

\ifsodaelse    {\usepackage{ltexpprt}\usepackage{amsmath}}{}
\usepackage{mathrsfs,amssymb,graphicx,verbatim}
\usepackage{paralist}

\newcommand\remove[1]{}

\newcommand{\rnote}[1]{}
\newcommand{\jnote}[1]{}

\newcommand{\e}{\varepsilon}
\newcommand{\R}{\mathbb{R}}
\renewcommand{\P}{\mathscr{P}}
\newcommand{\E}{\mathbb{E}}

\newcommand{\N}{\mathbb{N}}

\renewcommand{\P}{\mathscr{P}}

\newtheorem{theorem}{Theorem}[section]
\newtheorem{lemma}[theorem]{Lemma}

\newtheorem{corollary}[theorem]{Corollary}

\newtheorem{definition}[theorem]{Definition}

\newcommand{\eqdef}{\stackrel{\mathrm{def}}{=}}
\date{}

\renewcommand{\le}{\leqslant}
\renewcommand{\ge}{\geqslant}
\renewcommand{\leq}{\leqslant}
\renewcommand{\geq}{\geqslant}

\include{psfig}

\title{Scale-oblivious metric fragmentation and the nonlinear Dvoretzky theorem}

\author{Assaf Naor}
\address{Courant Institute, New York University, New York NY 10012}
\email{naor@cims.nyu.edu}
\author{Terence Tao}
\address{Department of Mathematics, UCLA, Los Angeles CA 90095-1555}
\email{tao@math.ucla.edu}
\date{}

\begin{document}

\begin{abstract}
We introduce a randomized iterative  fragmentation procedure for finite metric spaces, which is guaranteed to result in a polynomially large subset that is   $D$-equivalent to an ultrametric, where $D\in (2,\infty)$ is a prescribed target distortion. Since this procedure works for $D$ arbitrarily close to the nonlinear Dvoretzky phase transition at distortion $2$, we thus obtain a much simpler probabilistic proof of the main result of~\cite{BLMN05-main}, answering a question from~\cite{MN07}, and yielding the best known bounds in the nonlinear Dvoretzky theorem.

Our method utilizes a sequence of random scales at which a given metric space is fragmented. As in many previous randomized arguments in embedding theory, these scales are chosen irrespective of the geometry of the metric space in question. We show that our bounds are sharp if one utilizes such a ``scale-oblivious" fragmentation procedure.
\end{abstract}

\maketitle

\section{Introduction}\label{sec:intro}

A metric space $(X,d)$ is said to embed into Hilbert space with
\emph{distortion} $D\ge 1$ if there exists  $f:X\to \ell_2$
satisfying $d(x,y)\le \|f(x)-f(y)\|_2\le Dd(x,y)$ for all $x,y\in
X$. Dvoretzky's theorem~\cite{Dvo60} asserts that for every $k\in
\N$ and $D>1$ there exists $n=n(k,D)\in \N$ such that every
$n$-dimensional normed space has a $k$-dimensional linear subspace
that embeds into Hilbert space with distortion $D$; see~\cite{Mil71,MS99,Sch06} for the best known bounds on $n(k,D)$.

Motivated by a possible analogue of Dvoretzky's theorem in the class
of general metric spaces, Bourgain, Figiel and Milman introduced
in~\cite{BFM86} the {\em nonlinear Dvoretzky problem}, which asks
for the largest integer $k=k(n,D)$ such that any $n$-point metric
space has a subset of cardinality $k$ that embeds into Hilbert space
with distortion $D$. They showed~\cite{BFM86} that for every $D>1$
we have $\lim_{n\to\infty} k(n,D)=\infty$, thus establishing the
validity of a nonlinear Dvorezky phenomenon. Quantitatively, the main
result of~\cite{BFM86} asserts that $k(n,D)\ge c(D)\log n$, and that
there exists $D_0>1$ for which $k(n,D_0)=O(\log n)$.

Renewed interest in the nonlinear Dvorezky problem due to the
discovery of applications to the theory of online algorithms
resulted in a sequence of works~\cite{KKR94,BKRS00,BBM06} which culminated in the following
threshold phenomenon from~\cite{BLMN05-main} (see
also~\cite{BLMN05-dic,BLMN05-low,CK05} for related results):
\begin{theorem}[\cite{BLMN05-main}]\label{thm:blmn}
For $D>1$ there exist $a(D),A(D)\in (0,\infty)$ and $b(D),B(D)\in
(0,1)$ with the following properties:
\begin{enumerate}
\item If $D\in (1,2)$ then any $n$-point metric space has a subset of cardinality $\ge
a(D)\log n$ that embeds with distortion $D$ into Hilbert space. On the other hand, there exist
arbitrarily large $n$-point metric spaces $X_n$ with the property that any
$Y\subseteq X_n$ that embeds into Hilbert space with distortion $D$
necessarily satisfies $|Y|\le A(D)\log n$.
\item If $D\in (2,\infty)$ then any $n$-point metric space has a subset of cardinality $\ge
n^{1-b(D)}$ that embeds with distortion $D$ into Hilbert space. On the
other hand, there exist arbitrarily large $n$-point metric spaces
$X_n$ with the property that any $Y\subseteq X_n$ that embeds into Hilbert space
with distortion $D$ necessarily satisfies $|Y|\le n^{1-B(D)}$.
\end{enumerate}
\end{theorem}

Note that the first assertion of part $(1)$ of Theorem~\ref{thm:blmn} is just a restatement of the Bourgain-Figiel-Milman nonlinear Dvoretzky theorem~\cite{BFM86}.

All the positive embedding results quoted above are actually stronger than embeddings into Hilbert space: they produce subsets which embed with distortion $D$ into an {\em ultrametric}. Recall that a metric space $(U,\rho)$ is an ultrametric if  $\rho(u,v)\le \max\{\rho(u,w),\rho(w,v)\}$ for every $u,v,w\in U$. Separable ultrametrics embed isometrically into Hilbert space~\cite{VT79}, so the problem of finding a subset of a metric space which embeds with distortion $D$ into an ultrametric is a (strictly)  stronger statement than the nonlinear Dvoretzky problem. The fact that the embeddings of Theorem~\ref{thm:blmn} are into ultrametrics is crucial for its applications. Note, however, that the impossibility results in Theorem~\ref{thm:blmn} rule out embeddings into Hilbert space, and not just embeddings into ultrametrics.

The proofs of the nonlinear Dvoretzky theorems in~\cite{BFM86,KKR94,BKRS00,BBM06,BLMN05-dic,BLMN05-low,CK05} proceed via \emph{deterministic} constructions. In~\cite{MN07} a new approach to the nonlinear Dvoretzky problem was introduced, based on a \emph{probabilistic} argument which is closer in spirit to the proofs of the classical Dvoretzky theorem. This randomized approach, called the method of {\em Ramsey partitions}, has three main advantages. First, it leads to new algorithmic applications of the nonlinear Dvoretzky theorem which are very different from the applications in~\cite{KKR94,BKRS00,BBM06,BLMN05-main}; we shall briefly describe one of these applications in Section~\ref{sec:limitation}. Second, Ramsey partitions yield a major simplification of the proof of part $(2)$ in Theorem~\ref{thm:blmn} for sufficiently large values of $D$ (this part of Theorem~\ref{thm:blmn} is by far the most complicated part of its proof in~\cite{BLMN05-main}). Third, the bound on the exponent $b(D)$ obtained in~\cite{MN07} is asymptotically sharp as $D\to \infty$, unlike the bound in~\cite{BLMN05-main}, which is off by a logarithmic factor. Specifically, \cite{MN07} yields $b(D)\lesssim 1/D$, which is optimal up to the implied universal constant due to the bound $B(D)\gtrsim 1/D$ of~\cite{BLMN05-main}.

 An obvious question, raised in~\cite{MN07}, suggests itself: can the randomized approach of~\cite{MN07} yield a proof Theorem~\ref{thm:blmn} in which the target distortion $D>2$ is allowed to go all the way down to the phase transition at $2$? The main result of~\cite{MN07} states that for $D>2$, any $n$-point metric space $X$ has a subset $Y\subseteq X$ with $|Y|\ge n^{1-128/D}$ which embeds with distortion $D$ into an ultrametric. \cite{MN07} did not attempt to optimize the constant $128$ in this result, and indeed by a more careful analysis of the arguments of~\cite{MN07} one can ensure that $|Y|\ge n^{1-16/(D-2)}$ (even this estimate can be slightly improved, but not by much). In any case, it is clear that these statements become vacuous for $D$ smaller than a universal constant close enough to $2$. Thus, the full  $D>2$ range of part $(2)$ of Theorem~\ref{thm:blmn} still required the use of the deterministic approach of~\cite{BLMN05-main}.

 It was stated in~\cite{MN07} that there does not seem to be a simple way to use Ramsey partitions to handle distortions arbitrarily close to $2$. In Section~\ref{sec:limitation} we make a very simple observation which proves that if $D<3$, then the method of Ramsey partitions cannot yield a subset $Y$ as above of size tending to $\infty$ with $n$. Thus, in fact, it is {\em impossible} to approach the phase transition at $2$ using Ramsey partitions.
 Here we present a new randomized approach, building on the multiplicative telescoping argument of~\cite{MN07}, which proves the nonlinear Dvoretzky theorem for any distortion $D>2$. Specifically, we prove the following result:
 \begin{theorem}\label{thm:theta}
 For every $D>2$, any $n$-point metric space has a subset of cardinality $n^{\theta(D)}$ which embeds with distortion $D$ into an ultrametric. Here $\theta=\theta(D)\in (0,1)$ is the unique solution of the equation $$\frac{2}{D}=(1-\theta)\theta^{\frac{\theta}{1-\theta}}.$$
 \end{theorem}
 It is elementary to check that $\theta(D)\ge 1-\frac{2e}{D}$ for all $D>2$, and that as $\e\searrow 0$ we have $\theta(2+\e)=\frac{\e}{2\log(1/\e)}+O\left(\frac{\e\log\log(1/\e)}{(\log(1/\e))^2}\right)$. Theorem~\ref{thm:theta} yields a very short proof  (complete details in 3 pages) of the the nonlinear Dvoretzky theorem for all distortions $D>2$, with the best known bounds on the exponent $\theta(D)$. In a sense that is made precise in Section~\ref{sec:describe frag}, the above value of $\theta(D)$ is optimal for our method.

 \subsection{Approximate distance oracles and limitations of Ramsey partitions}\label{sec:limitation} We recall some terminology and results from~\cite{MN07}. Fix $\delta\in (0,1)$ and let $(X,d)$ be an $n$-point metric space of diameter $1$. A sequence $\{\P_k\}_{k=0}^\infty$ of partitions   of $X$ is called a \emph{partition tree} of rate $\delta$ if $\P_0$ is the trivial partition $\{X\}$, for all $k\ge 0$ $\P_{k+1}$ is a refinement of $\P_k$, and each set in $\P_k$ has diameter at most $\delta^k$.


The main tool in~\cite{MN07} is {\em random partition trees}. Let $\Pr$ be a probability distribution over partition trees of rate $\delta$. For $\ell>0$ consider the the random subset $Y\subseteq X$ consisting of those $x\in X$ such that for {\em all} $k\in \N$ the entire closed ball $B(x,\delta^k/\ell)$ is contained in the element of $\P_k$ to which $x$ belongs. Assume that each $x\in X$ falls in $Y$ with $\Pr$-probability at least $n^{-\beta}$. Then $\E\left[|Y|\right]\ge n^{1-\beta}$. Define for distinct $x,y\in X$ the random quantity $\rho(x,y)=\delta^{k(x,y)}$, where $k(x,y)$ is the largest integer $k$ such that both $x$ and $y$ fall in the same element of $\P_k$. Then $\rho$ is an ultrametric on $X$, and for $x\in X$ and $y\in Y$ we have $\rho(x,y)\ge d(x,y)\ge \frac{\delta}{\ell}\rho(x,y)$ \cite[Lem.\ 2.1]{MN07}. Thus,  on $Y$, $\rho$ is  bi-Lipschitz equivalent to the original metric $d$ with distortion $\le \ell/\delta$. But more is true: the ultrametric $\rho$ is defined on all of $X$, and approximates up to a factor $\le \ell/\delta$ all distances from points of $Y$ to {\em all the other points of $X$}.


 In~\cite{MN07} random partition trees were constructed with the desired bounds on $\beta$ and the distortion $\ell/\delta$. It was shown in~\cite{MN07} that the existence of an ultrametric $\rho$ on $X$ which has the above property of approximating distances from points of a large subset $Y\subseteq X$ to all other points of $X$, has a variety of implications to the theory of data structures. Here we need to briefly  recall the connection to {\em approximate distance oracles}.

An $n$-point metric space $(X,d)$ can be thought of as table of $\binom{n}{2}$ numbers, corresponding to the distances between all unordered pairs $x,y\in X$. In the approximate distance oracle problem the goal is, given $D>1$, to do ``one time work" (preprocessing) that produces a data structure (called an approximate distance oracle) of size $o(n^2)$ such that given a ``query" $x,y\in X$, one can quickly produce a number $E(x,y)$ satisfying $d(x,y)\le E(x,y)\le Dd(x,y)$. We call $D$ the {\em stretch} of the approximate distance oracle.

The seminal work on approximate distance oracles is due to Thorup and Zwick~\cite{TZ05}, who showed that for all odd $D\in \N$ one can design a data structure of size $O(Dn^{1+2/D})$ using which one can compute in time $O(D)$ a number $E(x,y)$ satisfying $d(x,y)\le E(x,y)\le Dd(x,y)$. In~\cite{MN07} it was shown\footnote{This assertion is not stated explicitly in~\cite{MN07}, but it follows directly from the proof of~\cite[Th.\ 1.2]{MN07}: using the notation of~\cite{MN07}, as noted in the proof of~\cite[Th.\ 1.2]{MN07}, the ultrametric $\rho_j$ is only required to be defined, and satisfy the conclusion of~\cite[Lem.\ 4.2]{MN07},  on $X_{j-1}$ and not on all of $X$. This property is guaranteed by our assumption. Thus there is no loss of constant factor since for the purpose of~\cite[Th.\ 1.2]{MN07} (unlike other applications of~\cite[Lem.\ 4.2]{MN07} in~\cite{MN07}), we do not need to use~\cite[Lem.\ 4.1]{MN07}.} that if every $n$-point metric space $(X,d)$ admits an ultrametric $\rho$ (defined on all of $X$) and a subset $Y\subseteq X$ with $|Y|\ge n^{1-c/D}$, such that for every $x\in X$ and $y\in Y$ we have $d(x,y)\le \rho(x,y)\le Dd(x,y)$, then any $n$-point metric space can be preprocessed to yield a data structure of size $O(n^{1+c/D})$ using which one can compute in time $O(1)$ a number $E(x,y)$ satisfying $d(x,y)\le E(x,y)\le Dd(x,y)$. A key new point here is that the query time is a universal constant, and does not depend on $D$ as in~\cite{TZ05}.

It was also shown in~\cite{TZ05} that any approximate distance oracle that answers distance queries with stretch $D<3$ must use $\gtrsim n^2$ bits of storage. Combining this lower bound with the above construction of~\cite{MN07}, we see that if $D<3$ there must exist arbitrarily large $n$-point metric spaces $(X_n,d_n)$ such that if $\rho$ is an ultrametric on $X_n$ and $Y\subseteq X_n$ is such that $d(x,y)\le \rho(x,y)\le Dd(x,y)$ for all $x\in X_n$ and $y\in Y$, then $|Y|\lesssim n^{o(1)}$. It is actually not difficult to unravel the arguments of~\cite{TZ05,MN07} to give a direct proof of the fact that Ramsey partitions cannot yield the nonlinear Dvoretzky theorem for distortions in $(2,3)$. We will not do so here since it would be a digression from the topic of the present paper; the purpose of the above discussion is only to explain why a method other than Ramsey partitions is required in order to to go all the way down to distortion $2$.

 \subsection{The fragmentation procedure and admissible exponents}\label{sec:describe frag}

Having realized that a proof of part $(2)$ of Theorem~\ref{thm:blmn} for $D$ arbitrarily close to $2$ cannot produce a large $Y\subseteq X$ and an ultrametric $\rho$ that is defined on all of $X$ and satisfies $d(x,y)\le \rho(x,y)\le Dd(x,y)$ for all $x\in X$ and $y\in Y$, it is natural to try to design a procedure which results in an ultrametric that is defined on the subset $Y$ alone. This is what our fragmentation procedure does.

In order to state our main results, we require the following definition:

\begin{definition}[Admissible exponent]\label{admis}  Fix $D >2$. We say that $\sigma > 0$ is an \emph{admissible exponent} for $D$ if there exist a sequence of (not necessarily independent) random variables
$ 1 = r_0 \ge r_1 \ge r_2 \ge \ldots > 0$
with $\lim_{n\to \infty} r_n=0$, such that
for every real number $r > 0$, we have
\begin{equation}\label{sigma-def}
 \sum_{n=1}^\infty \Pr\left[ r_n < r \le r_n + \frac{2r_{n-1}}{D} \right] \leq \sigma.
\end{equation}
\end{definition}
\begin{theorem}[Ultrametrics via admissible exponents]\label{coro:main}  Fix $D >2$, and let $\sigma > 0$ be an admissible exponent for $D$.  Let
$X = (X,d)$ be a finite metric space.  Then there exists a subset $S$ of $X$ of cardinality $|S| \geq |X|^{1-\sigma}$ which embeds with distortion  $D$ into an ultrametric.
\end{theorem}
Let $\sigma^*(D)$ denote the infimum of those $\sigma>0$ which are admissible exponents for $D$. Due to Theorem~\ref{coro:main}, we would like to estimate $\sigma^*(D)$. In fact, it turns out that we can compute it exactly; the following theorem, in combination with Theorem~\ref{coro:main}, implies Theorem~\ref{thm:theta}.

\begin{theorem}[Optimization of admissible exponents]\label{thm:compute admissible} For every $D>2$ we have $\sigma^*(D)=\beta$, where $\beta\in (0,1)$ is the unique solution of the equation
\begin{equation}\label{eq:def:beta}
\frac{2}{D}=\beta(1-\beta)^{\frac{1-\beta}{\beta}}.
\end{equation}
Moreover, $\sigma^*(D)$ is attained at the following random variables: $r_0=1$, and for $n\in \N$,
\begin{equation}\label{eq: r_n}
r_n\stackrel{\mathrm{def}}{=} (1-\beta)^{\frac{U+n-1}{\beta}},
\end{equation}
where $U$ is a random variable that is uniformly distributed on the interval $[0,1]$. For this choice of $ 1 = r_0 \ge r_1 \ge r_2 \ge \ldots > 0$, the supremum of the left hand side of~\eqref{sigma-def} over $r>0$ equals the value of $\beta$ in~\eqref{eq:def:beta}.
\end{theorem}

The construction of the subset $S$ in Theorem~\ref{coro:main} is most natural to describe in the context of compact metric spaces, though it will be applied here only to finite metric spaces. Throughout this paper a \emph{metric probability space} $(X,d,\mu)$ is a compact metric space $(X,d)$ equipped with a Borel probability measure $\mu$. For $x\in X$ and $r\ge 0$ we shall use the standard notation for (closed) balls: $B(x,r)=\{y\in X:\ d(x,y)\le r\}$.
To avoid degeneracies we assume that for every $r> 0$ we have $\mu(B(x,r)) > 0$, and that the function $x \mapsto \mu(B(x,r))$ is measurable.  Of course, these hypotheses are automatic in the case of finite metric spaces with uniform measure.

Fix a metric probability space $(X,d,\mu)$, normalized to have diameter $2$, and a decreasing sequence of radii $ 1 = r_0 \ge r_1 \ge r_2 \ge \ldots > 0.$ Define inductively a decreasing sequence of random subsets $X=S_0\supseteq S_1\supseteq S_2\ldots$ as follows. Having defined $S_i$, let $\{x_n\}_{n=1}^\infty$ be an i.i.d. sequence of points in $X$, each distributed according to $\mu$. The set $S_{i+1}$ is defined to be those points $x\in S_i$ for which the first point in the sequence $\{x_n\}_{n=1}^\infty$ that fell in $B\left(x,r_n+\frac{2r_{n-1}}{D}\right)$, actually fell in the smaller ball $B(x,r_n)$. Letting $S=\bigcup_{i=0}^\infty S_i$, we argue (see Lemma~\ref{lem:frag} and Lemma~\ref{lem:iterated}) that $(S,d)$ embeds with distortion $D$ into an ultrametric, and that,
\begin{equation}\label{eq:refine}
 \E\left[\mu(S)\right] \geq \int_X \left(\prod_{n=1}^\infty \frac{\mu(B(x,r_n))}{\mu\left(B\left(x,r_n + \frac{2r_{n-1}}{D}\right)\right)}\right)\ d\mu(x).
\end{equation}

So far we did not use the fact that the radii $\{r_n\}_{n=0}^\infty$ are themselves random. The additional randomness allows us to use a refinement of an idea of~\cite{MN07} in order to control the infinite product appearing in~\eqref{eq:refine} using Jensen's inequality (the corresponding step in~\cite{MN07} used the AM-GM inequality). This is how the notion of admissible exponent appears in Theorem~\ref{coro:main}; the details appear in Section~\ref{sec:frag}. Note that the proof of Theorem~\ref{thm:theta} is simple to describe: it follows the above outline with the specific sequence of random radii given in~\eqref{eq: r_n}. (Observe that this sequence of radii involves a choice of only one random number $U$, unlike the construction of~\cite{MN07}, and its predecessors~\cite{CKR04,FRT04},  in which $r_n$ was uniformly distributed on $[8^{-n}/4,8^{-n}/2]$, and the $\{r_n\}_{n=0}^\infty$ were independent random variables.)

The obvious weakness of the above approach is that the random radii $\{r_n\}_{n=0}^\infty$ are chosen without consideration of the particular geometry of the metric space $X$. It makes sense that in order to obtain sharper results one would need to investigate how different scales in $X$ interact, and reflect this understanding in a choice of radii which are not ``scale-oblivious". Theorem~\ref{thm:compute admissible} shows that in order to improve our bounds in Theorem~\ref{thm:theta} one would need to use a fragmentation procedure that is not scale-oblivious (or, find a way to control an expression such as~\eqref{eq:refine} without using Jensen's inequality; this seems quite difficult).

A particular question of interest in this context is as follows: for $D>2$ let $\theta^*(D)$ be the supremum of those $\theta>0$ such that there exists $n_0\in \N$ for which any metric space of cardinality $n\ge n_0$  has a subset of size $\ge n^\theta$ that embeds with distortion $D$ into an ultrametric. Both~\cite{BLMN05-main} and our new proof give the bound $\theta^*(2+\e)\gtrsim \e/\log(2/\e)$ (for different reasons). Must it be the case that $\theta^*(2+\e)$ tends to $0$ as $\e\searrow 0$? This is of course related the unknown behavior of the nonlinear Dvoretzky problem at distortion $D=2$. Computing the value of $\limsup_{D\to\infty} D(1-\theta^*(D))$ is also of interest; due to Theorem~\ref{thm:compute admissible} we know that using our scale-oblivious metric fragmentation procedure we cannot bound this number by less than $2e$.



\subsection*{Acknowledgements} We thank Manor Mendel for helpful discussions on the Thorup-Zwick lower bound. A.~N. is supported  by NSF grants  CCF-0635078
and CCF-0832795, BSF grant 2006009, and the Packard Foundation.
T.~T. is supported by a grant from the MacArthur foundation, by NSF grant
DMS-0649473, and by the NSF Waterman award.

\section{Randomized fragmentation}\label{sec:frag}


We begin with a lemma that fragments a metric space at a single pair  of scales $R>r>0$.

\begin{lemma}[Fragmentation lemma]\label{lem:frag}  Let $(X,d,\mu)$ be a metric probability space, and let $S\subseteq X$ be a compact subset of $X$. Fix $R > r > 0$ and a Borel-measurable non-negative function $w: S \to [0,\infty)$.  Then there exists a compact subset $T\subseteq S$ with
\begin{equation}\label{ep}
 \int_{T} \frac{\mu(B(x,R))}{\mu(B(x,r))} w(x)\ d\mu(x) \geq \int_S w(x)\ d\mu(x),
 \end{equation}
such that $T$ can be partitioned as $T = \bigcup_{n=1}^\infty T_n$, where each (possibly empty) $T_n$ is compact and contained in a ball of radius $r$, and any two non-empty $T_n, T_m$ are separated by a distance of at least $R-r$.
\end{lemma}

\begin{proof}
We use the probabilistic method.  Let $\{x_n\}_{n=1}^\infty$ be an i.i.d. sequence of points in $X$, selected using the measure $\mu$.  Observe that as $B(x,R)$ has positive measure for all $x\in X$, we will almost surely have $x_n \in B(x,R)$ for at least one $n \in \N$.  Thus if
we define the (random) quantity
\begin{equation}\label{nx}
 n(x) \stackrel{\mathrm{def}}{=} \inf \{ n \in \N: x_n \in B(x,R) \},
\end{equation}
then $n(x)$ is finite for almost every $x \in X$, and $x\mapsto n(x)$ is a measurable function of $x$.

Define a (random) subset $A \subseteq S$ by
\begin{equation}\label{epdef}
 A \stackrel{\mathrm{def}}{=} \{ x \in S: n(x)<\infty\ \wedge\   x_{n(x)} \in B(x,r) \}.
\end{equation}
Then $A = \bigcup_{n = 1}^\infty A_n$, where
\begin{equation}\label{epx}
 A_n \stackrel{\mathrm{def}}{=} \{ x \in S: n(x) = n \ \wedge\  x_n \in B(x,r) \}.
\end{equation}
By definition we have $A_n\subseteq B(x_n,r)$.  Also, if $x \in A_n$ and $y \in A_m$ for some $1 \leq n < m$, then by the definitions \eqref{nx}, \eqref{epx} we have $d(x_n,x)\le r$ and $d(x_n,y)>R$,
and hence by the triangle inequality we have $d(x,y) > R-r$.  Thus if we set $T_n \stackrel{\mathrm{def}}{=} \overline{A_n}$, then $T_n$ and $T_m$ are compact and separated by a distance of at least $R-r$ (this shows that only finitely many of the $T_n$ are non-empty). If we define $T \stackrel{\mathrm{def}}{=} \bigcup_{n = 1}^\infty T_n$, then $T$ is a compact subset of $S$.

Since $T\supseteq A$, in order to conclude the proof of Lemma~\ref{lem:frag} it suffices to prove the identity
\begin{equation}\label{eq:expectation}
\E \left[\int_A \frac{\mu(B(x,R))}{\mu(B(x,r))} w(x)\ d\mu(x) \right]= \int_S w(x)\ d\mu(x).
\end{equation}
By the Fubini-Tonelli theorem, in order to prove~\eqref{eq:expectation} it suffices
 to show that for all $x \in S$ we have,
\begin{equation}\label{eq:prob identity}
\Pr[ x \in A ] = \frac{\mu(B(x,r))}{\mu(B(x,R))}.
\end{equation}
Since $n(x)$ is finite almost surely, the definition~\eqref{epdef}, together with the joint independence of $x_1,x_2\ldots$, immediately implies that:
\begin{multline}\label{eq:compute prob}
\Pr[ x \in A ]= \sum_{n=1}^\infty \Pr\left[x_n \in B(x,r)\ \wedge\  x_1,\ldots,x_{n-1} \not \in B(x,R)\right]\\=\sum_{n=1}^\infty \mu(B(x,r))\left(1-\mu(B(x,R))\right)^{n-1}=\frac{\mu(B(x,r))}{\mu(B(x,R))}.
\end{multline}
This proves~\eqref{eq:prob identity}, and thus concludes the proof of Lemma~\ref{lem:frag}.
\end{proof}

We can iterate Lemma~\ref{lem:frag} as follows.

\begin{lemma}[Iterated fragmentation lemma]\label{lem:iterated} Fix $R,D>0$.
Let $(X,d,\mu)$ be a metric probability space of diameter at most $2R$, and let
$$ R = r_0 \ge r_1 \ge r_2 \ge \ldots > 0$$
be a sequence of radii converging to zero.    Then there exists a compact subset $S$ of $X$ such that
\begin{equation}\label{ep-mass}
 \mu(S) \geq \int_X \left(\prod_{n=1}^\infty \frac{\mu(B(x,r_n))}{\mu\left(B\left(x,r_n + \frac{2r_{n-1}}{D}\right)\right)}\right)\ d\mu(x),
\end{equation}
and $(S,d)$ embeds with distortion $D$ into an ultrametric.
\end{lemma}

\begin{proof} By applying Lemma~\ref{lem:frag} repeatedly, we obtain a decreasing sequence of compact subsets of $X$,
$$ X = S_0 \supseteq S_1 \supseteq S_2 \supseteq \ldots$$
satisfying for $n\ge 1$,
$$ \int_{S_n} \left(\prod_{m =n+1}^\infty \frac{\mu(B(x,r_m))}{\mu\left(B\left(x,r_m + \frac{2r_{m-1}}{D}\right)\right)}\right) \ d\mu(x) \geq \int_{S_{n-1}} \left(\prod_{m =n}^\infty \frac{\mu(B(x,r_m))}{\mu\left(B\left(x,r_m + \frac{2r_{m-1}}{D}\right)\right)}\right) \ d\mu(x),$$
such that for each $n\in \N$ we have $S_n= \bigcup_{j=1}^\infty S_{n,j}$, where each $S_{n,j}$ is compact and contained in a ball of radius $r_n$, and if $S_{n,j},S_{n,\ell}\neq \emptyset$ then $d(S_{n,j},S_{n,\ell})\ge 2r_{n-1}/D$. It follows inductively that
$$ \int_{S_n} \left(\prod_{m=n+1}^\infty \frac{\mu(B(x,r_m))}{\mu\left(B\left(x,r_m + \frac{2r_{m-1}}{D}\right)\right)}\right) \ d\mu(x) \geq
\int_X \left(\prod_{m=1}^\infty \frac{\mu(B(x,r_m))}{\mu\left(B\left(x,r_m + \frac{2r_{m-1}}{D}\right)\right)}\right)\ d\mu(x),$$
and in particular
$$ \mu(S_n) \geq \int_X \left(\prod_{m=1}^\infty \frac{\mu(B(x,r_m))}{\mu\left(B\left(x,r_m + \frac{2r_{m-1}}{D}\right)\right)}\right)\ d\mu(x).$$
If we set $S \stackrel{\mathrm{def}}{=} \bigcap_{n=1}^\infty S_n$, then $S$ is compact and obeys \eqref{ep-mass}.

If $x,y\in S$  are distinct, let $n(x,y)$ be the largest integer $n$ such that for all $m\in \{1,\ldots,n\}$ there is $j(m)\in \N$ for which $x,y\in S_{m,j(m)}$. Note that since the diameter of $S_{m,j}$ is at most $2r_m$, and $\lim_{m\to \infty }r_m=0$, such an $n$ must exist. Now define an ultrametric $\rho$ on $S$ by
$$ \rho(x,y) \stackrel{\mathrm{def}}{=} 2r_{n(x,y)}.$$
It is immediate to check that $\rho$ is symmetric, and obeys the ultratriangle inequality
$$\forall x,y,z\in S,\quad \rho(x,z) \leq \max\{ \rho(x,z), \rho(y,z) \}.$$

If $x,y\in S$ are distinct and $n=n(x,y)$, then by definition $x,y\in S_{n,j}$ for some $j\in \N$ and $x\in S_{n+1,k}$, $y\in S_{n+1,\ell}$, where $k\neq \ell$. Thus $d(x,y)\le \mathrm{diam}(S_{n,j})\le 2r_n=\rho(x,y)$ and $d(x,y)\ge d(S_{n+1,k},S_{n+1,\ell})\ge 2r_n/D=\rho(x,y)/D$. It follows that
the identity map from $(S,d)$ to $(S,\rho)$ has distortion at most $D$, completing the proof of Lemma~\ref{lem:iterated}.
\end{proof}


Now suppose that $(X,d)$ is a finite metric space, and that $\mu$ is the counting measure on $X$.  Then Lemma~\ref{lem:iterated}
specializes to
\begin{corollary}[Iterated fragmentation lemma, finite case]\label{finite} Fix $D,R>0$.
Let $X = (X,d)$ be a finite metric space of diameter at most $2R$, and let
$$ R = r_0 \ge r_1 \ge r_2 \ge \ldots > 0$$
be a sequence of radii converging to zero.    Then there exists a subset $S$ of $X$ such that
\begin{equation}\label{ep-mass-finite}
 |S| \geq \sum_{x \in X} \prod_{n=1}^\infty \frac{|B(x,r_n)|}{\left|B\left(x,r_n + \frac{2r_{n-1}}{D}\right)\right|},
\end{equation}
and $(S,d)$ embeds with distortion $D$ into an ultrametric.
\end{corollary}

The condition \eqref{ep-mass-finite} is difficult to work with.  However, using a random choice of $r_n$, and Jensen's inequality, one can obtain a more workable condition in terms of the notion of admissible  exponent as in Definition~\ref{admis}. This is contained in Theorem~\ref{coro:main}, which we are now in position to prove.



\begin{proof}[Proof of Theorem~\ref{coro:main}] By rescaling we may assume that $X$ has diameter at most $2$.  We let $r_0,r_1,\ldots$ be the random variables in Definition \ref{admis}, i.e., \eqref{sigma-def} holds for all $r>0$.  Applying Corollary \ref{finite} we thus obtain a (random) subset $S\subseteq X$ obeying \eqref{ep-mass-finite}, which embeds with distortion $D$ into an ultrametric.  Taking expectations we obtain
$$ \E \left[|S| \right]\geq
\sum_{x \in X} \E \left[\prod_{n=1}^\infty \frac{|B(x,r_n)|}{|B\left(x,r_n + \frac{2r_{n-1}}{D}\right)|}\right],$$
and hence by Jensen's inequality,
\begin{equation}\label{eq:after jensen} \E \left[|S|\right] \geq
\sum_{x \in X} \exp\left( \E \left[\sum_{n=1}^\infty \log \left(\frac{|B(x,r_n)|}{|B\left(x,r_n + \frac{2r_{n-1}}{D}\right)|}\right)\right]\right).\end{equation}

For every $x\in X$ let $0=t_1(x)<t_2(x)<\ldots<t_{k(x)}(x)$ be the radii at which $|B(x,t)|$ jumps, i.e.,
 $1=|B(x,t_1(x))|<|B(x,t_2(x))|<\ldots<|B(x,t_{k(x)}(x))|=|X|$, and
$B(x,t)=B(x,t_j(x))$ if $t_j(x)\le t<t_{j+1}(x)$ (where we use the convention $t_{k(x)+1}(x)=\infty$). Note that for every random variable $r\ge 0$ we have the following simple identity:
\begin{eqnarray}\label{eq:identity}
\E\left[\log |B(x,r)|\right]\nonumber&=&\sum_{j=1}^{k(x)} \Pr\left[t_j(x)\le r<t_{j+1}(x)\right]\log|B(x,t_j(x))|\\\nonumber&=&\sum_{j=1}^{k(x)} \left(\Pr\left[r\ge t_j(x)\right]-\Pr\left[r\ge t_{j+1}(x)\right]\right)\log|B(x,t_j(x))|\\&=&\sum_{j=2}^{k(x)}\Pr\left[r\ge t_j(x)\right]\log\left(\frac{|B(x,t_j(x))|}{|B(x,t_{j-1}(x))|}\right).
\end{eqnarray}
Applying~\eqref{eq:identity} to $r=r_n$ and $r=r_n+\frac{2r_{n-1}}{D}$, we see that~\eqref{eq:after jensen} can be written as
\begin{multline}\label{eq:rs}
\E \left[|S|\right] \\\geq\sum_{x \in X} \exp\left(-\sum_{j=2}^{k(x)}\left(\sum_{n=1}^\infty\Pr\left[r_n < t_j(x) \leq r_n + \frac{2r_{n-1}}{D}\right]\right)\log\left(\frac{|B(x,t_j(x))|}{|B(x,t_{j-1}(x))|}\right)\right).
\end{multline}
Applying \eqref{sigma-def} we conclude that
$$ \E \left[|S|\right] \geq \sum_{x \in X} \exp\left(- \sigma \sum_{j=2}^{k(x)} \log\left(\frac{|B(x,t_j(x))|}{|B(x,t_{j-1}(x))|}\right)\right)=\sum_{x\in X} |X|^{-\sigma}=|X|^{1-\sigma},$$
where we used the fact that $|B(x,t_1(x))|=1$ and $|B(x,t_{k(x)}(x))|=|X|$. The proof of Theorem~\ref{coro:main} is complete.
\end{proof}

\section{Proof of Theorem~\ref{thm:compute admissible}}\label{sec:addmissible}

Define $f:[0,1]\to [0,1]$ by $f(\beta)=\beta(1-\beta)^{\frac{1-\beta}{\beta}}$, where $f(0)=0$ and $f(1)=1$. Note that $(\log f)'(\beta)=-\frac{1}{\beta^2}\log(1-\beta)$, and therefore $f$ is strictly increasing on $[0,1]$. It follows that for each $\alpha\in [0,1]$ there is a unique $\beta=\beta(\alpha)$ satisfying the identity
\begin{equation}\label{eq:beta alpha}
\alpha=\beta(1-\beta)^{\frac{1-\beta}{\beta}}.
\end{equation}

Fix $D>2$ and set $\beta=\beta(2/D)$. Let $U$ be a random variable that is uniformly distributed on $[0,1]$. We shall define a sequence of random variables $r_0 \ge r_1 \ge r_2 \ge \ldots > 0$ as in~\eqref{eq: r_n}, i.e, by setting $r_0=1$, and for $n\in \N$,
$$
r_n\stackrel{\mathrm{def}}{=} (1-\beta)^{\frac{U+n-1}{\beta}}.
$$

Writing $\alpha =2/D$, for every $r>0$ we have the following bound on the left hand side of~\eqref{sigma-def}:
\begin{multline}\label{eq:get intervals}
 \sum_{n=1}^\infty \Pr\left[ r_n < r \le r_n + \frac{2r_{n-1}}{D} \right]\\\le  \sum_{n=1}^\infty \Pr\left[ (1-\beta)^{\frac{U+n-1}{\beta}} < r \le (1-\beta)^{\frac{U+n-1}{\beta}} + \alpha (1-\beta)^{\frac{U+n-2}{\beta}} \right]= \sum_{n=1}^\infty \Pr\left[ U\in I_n\right],
\end{multline}
where $I_n$ is the interval:
$$
I_n\stackrel{\mathrm{def}}{=}\left(\frac{\beta \log r}{\log (1-\beta)}-n+1,\frac{\beta \log r}{\log (1-\beta)}-n+1-\frac{\beta\log\left(1+\frac{\alpha}{(1-\beta)^{1/\beta}}\right)}{\log(1-\beta)}\right]\stackrel{\mathrm{def}}{=}(a-n,b-n].
$$
The identity~\eqref{eq:beta alpha} implies that $b-a=\beta$. In particular, $b-a\le 1$, and hence the intervals $\{I_n\}_{n=1}^\infty$ are disjoint, at most two of them intersect $[0,1]$, and the total length of the intersection of $\bigcup_{n=1}^\infty I_n$ with $[0,1]$ is at most $b-a$.
Combined with~\eqref{eq:get intervals}, this observation implies that
$$
\sum_{n=1}^\infty \Pr\left[ r_n < r \le r_n + \frac{2r_{n-1}}{D} \right]\le \mathrm{length}\left(\left(\bigcup_{n=1}^\infty I_n\right)\cap [0,1]\right)\le b-a=\beta.
$$

This proves the second assertion of Theorem~\ref{thm:compute admissible}. It remains to prove that for all $D>2$ we have $\sigma^*(D)\ge \beta(2/D)$.
To this end let $\{r_n\}_{n=0}^\infty$ be a sequence of random variables decreasing to zero as in Definition~\ref{admis}, so that \eqref{sigma-def} holds for some $\sigma>0$. Our goal is to show that $\sigma\ge  \beta(2/D)$.

For $\alpha,p\in (0,1)$ denote
\begin{equation}\label{eq:def beta_p}
\beta_p(\alpha)\eqdef \inf_{x> 1} \frac{(1+\alpha x)^p-1}{x^p-1}.
\end{equation}
By homogeneity, for every $y\ge x>0$ we have $(x+\alpha y)^p-x^p\ge \beta_p(\alpha)\left(y^p-x^p\right)$. Thus for all $n\in \N$,
\begin{equation}\label{eq:telescope}
\sum_{n=1}^\infty \E\left[\left(r_n+\frac{2r_{n-1}}{D}\right)^p-r_n^p\right]\ge \beta_p\left(\frac{2}{D}\right)\E\left[\sum_{n=1}^\infty \left(r_{n-1}^p-r_n^p\right)\right]=\beta_p\left(\frac{2}{D}\right),
\end{equation}
where we used the fact that $\lim_{n\to \infty} r_n=0$ and $r_0=1$.

Now,
\begin{multline}\label{eq:fub}
\sum_{n=1}^\infty \E\left[\left(r_n+\frac{2r_{n-1}}{D}\right)^p-r_n^p\right] =\sum_{n=1}^\infty \int_{0}^\infty pr^{p-1}\left(\Pr\left[r_n+\frac{2r_{n-1}}{D}\ge r\right]-\Pr\left[r_n\ge r\right]\right)dr\\
= \sum_{n=1}^\infty \int_0^{1+\frac{2}{D}}pr^{p-1}\Pr\left[ r_n < r \le r_n + \frac{2r_{n-1}}{D} \right]dr\stackrel{\eqref{sigma-def}}{\le} \sigma\left(1+\frac{2}{D}\right)^p.
\end{multline}
By combining~\eqref{eq:telescope} and~\eqref{eq:fub} (which hold for all $p\in (0,1)$), we see that the bound $\sigma\ge  \beta(2/D)$ will be proven if we manage to show that for all $\alpha\in (0,1)$,
\begin{equation}\label{eq:limit goal}
\limsup_{p\to 0} \beta_p(\alpha)\ge \beta(\alpha),
\end{equation}
where $\beta(\alpha)$ is the unique $\beta\in (0,1)$ satisfying~\eqref{eq:beta alpha}.

To prove~\eqref{eq:limit goal}, define $f:[0,\infty)\to \R$ by $f(x)=(1+\alpha x)^p-1-\beta_p(\alpha)\left(x^p-1\right)$. Note that $f$ takes only non-negative values, due to the definition~\eqref{eq:def beta_p}. By considering the limit as $x\to \infty$ of the right hand side of~\eqref{eq:def beta_p}, we see that $\beta_p(\alpha)\le \alpha^p$. But, it cannot be the case that $\beta_p(\alpha)=\alpha^p$, since otherwise $f(x)=(1+\alpha x)^p-(\alpha x)^p-(1-\alpha^p)$, which, since $p\in (0,1)$, tends to $-(1-\alpha^p)<0$ as $x\to\infty$, contradicting the non-negativity of $f$ on $[1,\infty)$. Thus $\beta_p(\alpha)<\alpha^p$. It follows in particular that the infimum in~\eqref{eq:def beta_p} is actually a minimum, i.e., there exists $x_0\in (1,\infty)$ for which $\beta_p(\alpha)=\frac{(1+\alpha x_0)^p-1}{x_0^p-1}$. This is the same as $f(x_0)=0$, and since $f$ is non-negative, $x_0$ must be a global minimum of $f$, and hence $f'(x_0)=0$.

From $0=f'(x_0)=p\alpha(1+\alpha x_0)^{p-1}-p\beta_p(\alpha)x_0^{p-1}$ we see that
\begin{equation}\label{eq:compute x_0}
x_0=\frac{1}{\left(\alpha/\beta_p(\alpha)\right)^{1/(1-p)}-\alpha}.
\end{equation}
Substituting this value of $x_0$ into the equation $f(x_0)=0$, we see that
\begin{equation}\label{eq:beta_p identity}
\frac{\left(\alpha/\beta_p(\alpha)\right)^{p/(1-p)}}
{\left(\left(\alpha/\beta_p(\alpha)\right)^{1/(1-p)}-\alpha\right)^p}-1=\beta_p(\alpha)
\left(\frac{1}
{\left(\left(\alpha/\beta_p(\alpha)\right)^{1/(1-p)}-\alpha\right)^p}-1\right).
\end{equation}

Denote $\beta=\limsup_{p\to 0} \beta_p(\alpha)$.  If $\beta=1$ then~\eqref{eq:limit goal} holds trivially. We may therefore assume that $\beta<1$. Moreover, \eqref{eq:compute x_0} combined with $x_0>1$ implies that $\left(\alpha/\beta_p(\alpha)\right)^{1/(1-p)}-\alpha<1$, or $\beta_p(\alpha)>\frac{\alpha}{(1+\alpha)^{1-p}}$. Thus $\beta\ge \frac{\alpha}{1+\alpha}$ (all that we will need below is that $\beta\neq 0$).

If $\{p_k\}_{k=1}^\infty \subseteq (0,1)$ is such that $\lim_{k\to \infty}p_k=0$ and $\lim_{k\to \infty}\beta_{p_k}(\alpha)=\beta$, then it follows from~\eqref{eq:beta_p identity} that:
\begin{equation}\label{eq:before dividing}
p_k\left(\log\left(\frac{\alpha}{\beta}\right)-\log\left(\frac{\alpha}{\beta}-\alpha\right)\right)+o(p_k)=
-p_k\beta\log\left(\frac{\alpha}{\beta}-\alpha\right)+o(p_k).
\end{equation}
Since, as argued above, $\frac{\alpha}{\beta}, \frac{\alpha}{\beta}-\alpha\in (0,\infty)$, the asymptotic identity~\eqref{eq:before dividing} implies that
$$
\log\left(\frac{\alpha}{\beta}\right)-\log\left(\frac{\alpha}{\beta}-\alpha\right)=
-\beta\log\left(\frac{\alpha}{\beta}-\alpha\right),
$$
which simplifies to give $\alpha=\beta(1-\beta)^{(1-\beta)/\beta}$. Since we already argued (in the paragraph preceding~\eqref{eq:beta alpha}), that $\beta(\alpha)$ is the unique solution of the equation~\eqref{eq:beta alpha}, we deduce that $\beta=\beta(\alpha)$. The proof of~\eqref{eq:limit goal}, and hence also the proof of Theorem~\ref{thm:compute admissible}, is complete.
\qed



\bibliographystyle{abbrv}
\bibliography{frag}

\end{document}
\endinput

\newpage
For $\e\in (0,1)$ denote $r_n^\e\eqdef \max\{r_n,\e\}$.
It follows from Lemma~\ref{lem:numerical} that
\begin{multline}\label{telescope}
\sum_{n=1}^\infty\E\left[\log\left(1+\frac{2r_{n-1}^\e}{Dr_n^\e}\right)\right]\ge \beta\left(\frac{2}{D}\right)\sum_{n=1}^\infty\E\left[\log\left(\frac{r_{n-1}^\e}{r_{n}^\e}\right)\right]\\=
\beta\left(\frac{2}{D}\right)\E\left[\sum_{n=1}^\infty \left(\log \left(r_{n-1}^\e\right)-\log \left(r_n^\e\right)\right)\right]=\beta\left(\frac{2}{D}\right)\log\left(\frac{1}{\e}\right).
\end{multline}
At the same time,
\begin{eqnarray}\label{eq:fubini}
\E\left[\log\left(1+\frac{2r_{n-1}^\e}{Dr_n^\e}\right)\right]\nonumber&=&
\E\left[\log\left(r_n^\e+\frac{2r_{n-1}^\e}{D}\right)\right]-\E\left[\log \left(r_n^\e\right)\right]\\\nonumber&=&\int_0^\infty \frac{1}{r}\Pr\left[r_n^\e+\frac{2r_{n-1}^\e}{D}\ge r\right]dr-\int_0^\infty \frac{1}{r}\Pr\left[r_n^\e\ge r\right]dr\\&\le&
\int_{\e\left(1+\frac{2}{D}\right)}^{1+\frac{2}{D}}\frac{1}{r}\Pr\left[r_n^\e<r\le  r_n^\e+\frac{2r_{n-1}^\e}{D} \right]dr.
\end{eqnarray}
It follows that
\begin{multline}
\sum_{n=1}^\infty \E\left[\log\left(1+\frac{2r_{n-1}^\e}{Dr_n^\e}\right)\right]\stackrel{\eqref{eq:fubini}}{=}
\int_\e^{1+\frac{2}{D}}\frac{1}{r}\left(\sum_{n=1}^\infty\Pr\left[r_n^\e<r\le  r_n^\e+\frac{2r_{n-1}^\e}{D} \right]\right)dr\\=\lim_{N\to \infty} \int_\e^{1+\frac{2}{D}}\frac{1}{r}\left(\sum_{n=1}^N\Pr\left[r_n<r\le  r_n^\e+\frac{2r_{n-1}^\e}{D} \right]\right)dr
\end{multline}

\begin{lemma}\label{lem:numerical} For every $\alpha\in (0,1)$, $p>0$ and $x\ge 1$, we have:
\begin{equation}\label{eq:numerical}
(1+\alpha x)^p-1\ge p\beta(\alpha)\log x.
\end{equation}
\end{lemma}
\begin{proof} Since $(1+\alpha x)^p-1=e^{p\log(1+\alpha x)}-1\ge p\log(1+\alpha x)$, it suffices to show that $g(x)\ge 0$ for $x\ge 1$, where $g(x)\eqdef \log(1+\alpha x)-\beta(\alpha)\log x$. Then $g(1)\ge 0$ and $\lim_{x\to \infty} g(x)=\infty$, since $\beta(\alpha)<1$. Now, $g'(x)=\frac{\alpha}{1+\alpha x}-\frac{\beta(\alpha)}{x}$, so the unique zero of $g'$ is at $x=\frac{\beta(\alpha)}{\alpha(1-\beta(\alpha))}$. It therefore suffices to check that  $g\left(\frac{\beta(\alpha)}{\alpha(1-\beta(\alpha))}\right)\ge 0$. But,
\begin{multline*}
g\left(\frac{\beta(\alpha)}{\alpha(1-\beta(\alpha))}\right)=\log\left(1+\frac{\alpha \beta(\alpha)}{\alpha(1-\beta(\alpha))}\right)-\beta(\alpha)\log\left(\frac{\beta(\alpha)}{\alpha(1-\beta(\alpha))}\right)\\
= \log\left(\frac{1}{1-\beta}\right)-\log\left(\frac{\beta(\alpha)^{\beta(\alpha)}}{\alpha^{\beta(\alpha)}
(1-\beta(\alpha))^{\beta(\alpha)}}\right)=\log\left(\frac{\alpha^{\beta(\alpha)}}{\beta(\alpha)^{\beta(\alpha)}
(1-\beta(\alpha))^{1-\beta(\alpha)}}\right)\stackrel{\eqref{eq:beta alpha}}{=}0,
\end{multline*}
as required.
\end{proof}

\subsection{Line into ultrametrics}

\begin{lemma} Fix $D>1$.
Let $S\subseteq \{1,\ldots,n\}$ be $D$-equivalent to an ultrametric. Then $|S|\le n^\alpha$, where
$$
\alpha =\frac{\log 2}{\log\left(\frac{2}{1-1/D}\right)}=1-\frac{1}{D\log 2}+O\left(\frac{1}{D^2}\right).
$$
\end{lemma}

\begin{proof}
The proof is by induction on $n$. The case $n=1$ is vacuous. Assume that $n>1$ and let $S\subseteq \{1,\ldots,n\}$ be $D$-equivalent to an ultrametric $\rho$:
 $$
 \forall i,j\in S,\quad |i-j|\le \rho(i,j)\le D|i-j|.
 $$
 Denote $a=\min S$ and $b=\max S$. We claim that there exists $p,q\in \{a,a+1,\ldots,b\}$ such that $[p,q]\cap S=\emptyset$ and $|q-p|\ge \frac{b-a}{D}$. Indeed, otherwise we could find $a=s_0<s_2<\ldots<s_k=b$ such that $\{s_0,\ldots,s_k\}\subseteq S$, and $|s_i-s_{i-1}|< \frac{b-a}{D}$ for all $i\in \{1,\ldots,k\}$.
\end{proof}
